\documentclass[11pt,twoside,a4paper,reqno]{amsart}

\usepackage{amsmath}
\usepackage{amsthm}
\usepackage{amsfonts, amssymb}
\usepackage[all]{xy}
\usepackage{url}
\usepackage{algorithm}
\usepackage[end]{algpseudocode}
\usepackage{mathtools,amssymb}
\usepackage{tikz}
\usepackage{fontawesome}
\usepackage{float}

\usepackage{latexsym}

\usepackage{graphicx}
\usepackage{graphics}
\usepackage[scriptsize,bf]{caption}
\usepackage{float}
\usepackage{enumerate}
\usepackage{verbatim}
\usepackage{color}
\usepackage{wrapfig}
\usepackage{subfig}

\theoremstyle{plain}
\newtheorem{lema}{\sc Lemma}
\newtheorem{prop}[lema]{\sc Proposition}
\newtheorem{teo}[lema]{\sc Theorem}

\newtheorem{coro}[lema]{\sc Corollary}

\theoremstyle{definition}
\newtheorem*{defi}{\sc Definition}
\newtheorem{ej}[lema]{\sc Example}

\newtheorem{preg}[lema]{\sc Question}

\newtheorem{obs}[lema]{\sc Remark}

\makeatletter
\@namedef{subjclassname@2020}{%
  \textup{2020} Mathematics Subject Classification}
  
\newcommand{\multiline}[1]{%
  \begin{tabularx}{\dimexpr\linewidth-\ALG@thistlm}[t]{@{}X@{}}
    #1
  \end{tabularx}
}    
\makeatother

\begin{document}

\title[Playing through a noisy channel]{Playing through a noisy channel (and knowing it)}

\author[N. A. Capitelli]{Nicol\'as A. Capitelli}

\address[N. A. Capitelli]{Universidad Nacional de Luj\'an, Departamento de Ciencias B\'asicas, Argentina}

\email{\color{blue}ncapitelli@unlu.edu.ar}

\author[M. L. Privitelli]{Melina L. Privitelli}

\address[M. L. Privitelli]{
Instituto de Tecnología e Ingeniería, Universidad Nacional de Hurlingham, Av. Gdor.
Vergara 2222, Villa Tesei, B1688GEZ, Buenos Aires, Argentina \& National Council of Science and Technology (CONICET), Argentina.}

\email{\color{blue}melina.privitelli@unahur.edu.ar}

\subjclass[2020]{91A46,91A05,91A60}

\keywords{Combinatorial games, noisy channels, move error functions, Nim, optimal strategies, winning probabilities}


\begin{abstract}
We study combinatorial games played through a noisy channel, where the move chosen by a player may differ from the move that is effectively executed. The key feature of this setting is that any received legal move is accepted as the move played, regardless of the player's original intention, while illegal moves trigger retransmission. Both players know the error mechanism and strategically choose their transmitted moves to maximize their winning probabilities. This interaction between strategic choice and imperfect execution gives rise to a rich class of games whose behavior can differ substantially from that of the underlying combinatorial game. We develop a general framework based on move error functions and establish structural results concerning winning probabilities, optimal moves, and conditions for fair chance games. We then investigate in detail one-pile \emph{Nim} with independent binary transmission errors. Despite its apparent simplicity, this example exhibits unexpectedly intricate behavior. We determine optimal strategies and winning probabilities in the low-error regime and uncover a substantially richer structure in the high-error regime, including changes in optimal play and nontrivial limiting behavior.
\end{abstract}

\maketitle

\section{Introduction} \label{Sec:Intro}

This article studies combinatorial games in which the move selected by a player is not necessarily the move that is effectively played. Let $G$ be a combinatorial game and suppose that both players encode their moves in the same way and transmit them through a noisy channel that can introduce errors during the transmission. The players are aware of this interference and accept the received move as the valid move whenever it is a legal move of the original game; otherwise, retransmission is requested. The probability of sending a move from a given position and receiving a possibly different one is known to the players. Thus, a player does not only take into account the winning and losing structure of $G$, but also the probability that a transmitted move is transformed into each of the other legal moves. In particular, a move which is not optimal in the original game may become optimal in the presence of noise.

A natural motivation for this setting arises in games controlled through imperfect interfaces, such as brain-computer interfaces, where the command recognized by the system may differ from the player's intended move. For example, Maby et al. \cite{Maby2012} developed a two-player version of \emph{Connect Four} controlled by brain activity, illustrating the relevance of action-recognition errors in competitive play. Although the specific interaction rules may differ, such systems motivate the study of how imperfect move execution affects strategic decisions. More broadly, the framework developed here may provide insights into other settings involving imperfect action execution, such as human--robot interaction, remote-controlled systems, and economic games in which intended actions may not be implemented accurately.

Randomness in games has long been studied, beginning with the stochastic-game framework of Shapley \cite{Shapley}, where the actions selected by the players determine probability distributions over subsequent states. More closely related to our setting are games with incompetent players, where stochastic matrices model errors in the execution of chosen strategies \cite{BeckFilar,BeckEjovFilar}; see also the survey \cite{GrahamKleshninaFilar}. Other models of imperfect execution include \cite{ArchibaldShoham,FarinaGatti}.

Randomness has also been incorporated directly into combinatorial games. Krawec \cite{Krawec} studies impartial combinatorial games in which one or more players choose moves randomly, including a model where a player makes a mistake with a given probability and otherwise plays strategically. Variants of \emph{Nim} with partial information, rather than noisy execution, were considered by van den Bergh, Kosters and Spieksma \cite{NimVariants}. Unlike the randomized-choice models, the setting considered here keeps both players strategic: a player deliberately chooses a legal move, and randomness acts only afterwards, possibly transforming the intended move into a different legal one.

From the perspective of stochastic game theory, our model can be viewed as a particular turn-based stochastic game, in which each selected move determines a probability distribution over the legal followers. Our focus, however, is on how a prescribed error mechanism affects optimal strategies and winning probabilities in combinatorial games. We start with the graph of a fixed combinatorial game and equip it with a move error function describing the probability of receiving one legal move when another one was transmitted. This allows us to compare directly the original game with its noisy versions and to study how winning probabilities and optimal moves depend on the error mechanism.

This paper develops this approach by combining motivating examples,
a general mathematical framework, and a detailed study of a
particular game. In Section~\ref{Sec:MotivatingExamples}, we
examine two classical combinatorial games, namely 1-pile
\emph{Nim} and \emph{Chomp!}, to illustrate how different
error mechanisms affect winning probabilities and optimal play.
These examples also show how transmission errors can turn
combinatorial games into fair or unfair chance games.
In Section~\ref{Sec:Framework}, we introduce move error
functions $\psi$ and the corresponding games $G_\psi$, and derive
a general recursive formula for determining optimal winning
probabilities. We establish structural characterizations
of positions with winning probability $0$ or $1$, and
prove stability results for finite games under small
perturbations of the error mechanism. We also obtain
sufficient conditions for fair chance games, with an
application to multi-pile \emph{Nim} under uniform move errors. In Section~\ref{Sect:ToyExampleRevisited}, we return to
1-pile \emph{Nim}, where moves are encoded in binary and
transmission errors occur independently in each digit.
For error probability $p\leq\frac{1}{2}$, we prove that
transmitting $0$ is optimal (Theorem~\ref{Teo:LowError})
and obtain an exact formula for the winning probabilities
(Proposition~\ref{Prop:ExactLowError}). We also show that,
for $0<p<\frac{1}{2}$, the first-player advantage
vanishes as the heap size grows
(Theorem~\ref{Teo:LowErrorLimit}).
The high-error regime $p>\frac{1}{2}$ reveals a substantially
richer structure. We completely classify the limiting
winning probabilities as $p\to1^-$
(Theorem~\ref{Teo:LimitHighError}) and uncover unexpected
optimal strategies. In particular, for heaps of size
$2^m$ with $m\geq8$, transmitting $2^m-4$ is uniquely
optimal when $p$ is sufficiently close to $1$
(Corollary~\ref{Coro:PowerTwoHighError}).
We also establish quantitative lower bounds for winning
probabilities in the high-error regime
(Theorem~\ref{Teo:PowerTwoHighError}).
Finally, we prove that, for each fixed heap size,
the winning probability is a continuous piecewise
rational function of $p$, so optimal moves can change
only finitely many times.

The main reference for combinatorial games is \cite{Conway} (see also \cite[\S 1]{Ferguson}). All games considered throughout the paper are impartial and follow the normal-play convention.

\section{Motivating Examples}\label{Sec:MotivatingExamples}

\subsection{1-pile \textit{Nim}}

First consider 1-pile \emph{Nim} with three chips. We encode the number
of chips in the heap using binary notation, so the initial position is
represented by \texttt{11}. A move is transmitted by specifying the
number of chips left in the heap after the move. Thus, from position
\texttt{11}, the legal moves are \texttt{10}, \texttt{01}, and
\texttt{00}. 
During transmission, each binary digit may be flipped
independently of the others, so that a transmitted
\texttt{0} may be received as \texttt{1}, and vice versa.
If the received word does not represent a legal move
from the current position, retransmission is requested.

We give a preliminary analysis of this game to illustrate
the method that will be used throughout the paper.
Let $p$ denote the probability that a given bit is flipped
during transmission, and let $N_3(p,s)$ denote the
probability that Player I wins from a heap of three chips
when the transmitted move is $s$.
 Suppose first that
Player I attempts to remove the entire pile, and hence transmits
\texttt{00}. Since the legal moves from \texttt{11} are \texttt{00},
\texttt{01}, and \texttt{10}, after conditioning on receiving a legal
move, the probabilities of receiving \texttt{00}, \texttt{01}, and
\texttt{10} are respectively
\[
\frac{1-p}{1+p},\qquad
\frac{p}{1+p},\qquad
\frac{p}{1+p}.
\]
Therefore, the winning probability $N_{\texttt{11}}(p,\texttt{00})$ that Player I
wins when transmitting \texttt{00} is
\begin{eqnarray}\label{Eq:1Heap}
N_{\texttt{11}}(p,\texttt{00})&=&P(\text{I wins }|\text{ II receives \texttt{00}})\cdot \left(\dfrac{1-p}{1+p}\right)+\nonumber\\
&&P(\text{I wins }|\text{ II receives \texttt{01}})\cdot \left(\dfrac{p}{1+p}\right)+  \\
&&P(\text{I wins }|\text{ II receives \texttt{10}})\cdot \left(\dfrac{p}{1+p}\right)\nonumber\end{eqnarray}
Clearly,
$P(\text{I wins}\mid\text{II receives }\texttt{00})=1$, whereas
$P(\text{I wins}\mid\text{II receives }\texttt{01})=0$, since from
position \texttt{01} Player II can remove the remaining chip.

It remains to determine the winning probability after Player II receives
\texttt{10}. From position \texttt{10}, Player II may transmit either
\texttt{00} or \texttt{01}. If Player II transmits \texttt{00}, then
\vspace{0.1in}
\begin{eqnarray*}
P(\text{I wins }|\text{ II receives \texttt{10}})&=&P((\text{I wins }|\text{ II receives \texttt{10}})\big|\text{ I receives \texttt{00}})\cdot P(\text{I receives \texttt{00}})\\
&+&P((\text{I wins }|\text{ II receives \texttt{10}})\big|\text{ I receives \texttt{01}})\cdot P(\text{I receives \texttt{01}})\\&=&0.(1-p)+1.p\\
&=& p.
\end{eqnarray*}

\vspace{0.1in}

\noindent Analogously, if Player II transmits \texttt{01}, then Player I wins with
probability $1-p$. Since Player II chooses the move that minimizes
Player I's probability of winning, we obtain
\[
P(\text{I wins}\mid\text{II receives }\texttt{10})
=
\begin{cases}
p, & 0\leq p\leq\frac12,\\
1-p, & \frac12\leq p<1.
\end{cases}
\] Substituting this into Equation~\eqref{Eq:1Heap}, we obtain the probability of winning
for Player I when \texttt{00} is transmitted:

\begin{equation*}
N_{\texttt{11}}(p,\texttt{00})=\begin{cases}\dfrac{1-p+p^2}{1+p}&   p\leq \frac{1}{2}\\
1-p&   \frac{1}{2}\leq p<1.
\end{cases}
\end{equation*}

\vspace{0.1in}

In the same way, we compute the probability of winning for Player I when
transmitting \texttt{01} or \texttt{10}:

\vspace{0.1in}

\begin{minipage}{0.49\textwidth}
\begin{equation*}
N_{\texttt{11}}(p,\texttt{01})=\begin{cases}p& p\leq \frac{1}{2}\\
\dfrac{p-p^3}{1-p+p^2}& \frac{1}{2}\leq p<1
\end{cases}
\end{equation*}\end{minipage}\begin{minipage}{0.49\textwidth}
\begin{equation*}
N_{\texttt{11}}(p,\texttt{10})=\begin{cases}\dfrac{2p-3p^2+p^3}{1-p+p^2}& p\leq \frac{1}{2}\\
1-p& \frac{1}{2}\leq p<1.
\end{cases}
\end{equation*}\end{minipage}

\vspace{0.1in}

\noindent Taking the maximum over the three possible transmitted moves, we obtain
the winning probability from position \texttt{11}: 
%
%
\begin{equation}\label{Eq:Nim}
N_{\texttt{11}}(p)=\left\{\begin{array}{cll}\dfrac{1-p+p^2}{1+p}& p\leq \frac{1}{2}&\text{(}\texttt{00}\text{)}\\
\dfrac{p-p^3}{1-p+p^2}& \frac{1}{2}\leq p<1&\text{(}\texttt{01}\text{)}.
\end{array}\right.
\end{equation}

\vspace{0.1in}

\noindent The corresponding optimal transmitted moves are indicated in the
rightmost column of Equation~\eqref{Eq:Nim}. Notice that when $p=\frac{1}{2}$,
all three transmitted moves induce the same distribution on the legal
followers, and the resulting game is fair. The same recursive procedure
can be applied to a heap of arbitrary size; we will return to the general case
in Section~\ref{Sect:ToyExampleRevisited}.

\subsection{Chomp!} Consider now \emph{Chomp!}. For a fixed targeted square, let
$p\in[0,1]$ denote the probability of hitting that square accurately,
and let $p_{ij}\in[0,1]$ denote the probability of hitting instead the
square in row $i$ and column $j$, where the sum is taken over the
non-targeted squares and
\[
p+\sum_{i,j}p_{ij}=1.
\]
Figure~\ref{Figure:Chomp} shows several examples of possible probability
distributions on an $n\times m$ bar. 
In this setting, an error in the execution of a move
may occur when a player aims at a given square
but the bar breaks at a different one.

\begin{figure}
\begin{minipage}{0.49\textwidth}
\begin{center}
\begin{tikzpicture}[scale=0.85]

\fill[gray] (0,0) rectangle (1,1);
\draw (0.5,0.5) node {\faBomb};

\draw (4.5,2.5) node {\small $p$};
\draw (3.5,2.5) node {\small $\frac{1-p}{8}$};
\draw (3.5,1.5) node {\small $\frac{1-p}{8}$};
\draw (3.5,3.5) node {\small $\frac{1-p}{8}$};
\draw (5.5,2.5) node {\small $\frac{1-p}{8}$};
\draw (5.5,1.5) node {\small $\frac{1-p}{8}$};
\draw (5.5,3.5) node {\small $\frac{1-p}{8}$};
\draw (4.5,1.5) node {\small $\frac{1-p}{8}$};
\draw (4.5,3.5) node {\small $\frac{1-p}{8}$};

\draw (0,0) rectangle (1,1);
\draw (1,0) rectangle (2,1);
\draw (2,0) rectangle (3,1);
\draw (3,0) rectangle (4,1);
\draw (4,0) rectangle (5,1);
\draw (5,0) rectangle (6,1);
\draw (6,0) rectangle (7,1);
\draw (7,0) rectangle (8,1);

\draw (0,1) rectangle (1,2);
\draw (1,1) rectangle (2,2);
\draw (2,1) rectangle (3,2);
\draw (3,1) rectangle (4,2);
\draw (4,1) rectangle (5,2);
\draw (5,1) rectangle (6,2);
\draw (6,1) rectangle (7,2);
\draw (7,1) rectangle (8,2);

\draw (0,2) rectangle (1,3);
\draw (1,2) rectangle (2,3);
\draw (2,2) rectangle (3,3);
\draw (3,2) rectangle (4,3);
\draw (4,2) rectangle (5,3);
\draw (5,2) rectangle (6,3);
\draw (6,2) rectangle (7,3);
\draw (7,2) rectangle (8,3);

\draw (0,3) rectangle (1,4);
\draw (1,3) rectangle (2,4);
\draw (2,3) rectangle (3,4);
\draw (3,3) rectangle (4,4);
\draw (4,3) rectangle (5,4);
\draw (5,3) rectangle (6,4);
\draw (6,3) rectangle (7,4);
\draw (7,3) rectangle (8,4);

\end{tikzpicture}\end{center}\end{minipage}
\begin{minipage}{0.49\textwidth}
\begin{center}
\begin{tikzpicture}[scale=0.85]

\fill[gray] (0,0) rectangle (1,1);
\draw (0.5,0.5) node {\faBomb};

\draw (4.5,2.5) node {\small $p$};
\draw (3.5,2.5) node {\small $\frac{1-p}{4}$};
\draw (5.5,2.5) node {\small $\frac{1-p}{4}$};
\draw (4.5,1.5) node {\small $\frac{1-p}{4}$};
\draw (4.5,3.5) node {\small $\frac{1-p}{4}$};

\draw (0,0) rectangle (1,1);
\draw (1,0) rectangle (2,1);
\draw (2,0) rectangle (3,1);
\draw (3,0) rectangle (4,1);
\draw (4,0) rectangle (5,1);
\draw (5,0) rectangle (6,1);
\draw (6,0) rectangle (7,1);
\draw (7,0) rectangle (8,1);

\draw (0,1) rectangle (1,2);
\draw (1,1) rectangle (2,2);
\draw (2,1) rectangle (3,2);
\draw (3,1) rectangle (4,2);
\draw (4,1) rectangle (5,2);
\draw (5,1) rectangle (6,2);
\draw (6,1) rectangle (7,2);
\draw (7,1) rectangle (8,2);

\draw (0,2) rectangle (1,3);
\draw (1,2) rectangle (2,3);
\draw (2,2) rectangle (3,3);
\draw (3,2) rectangle (4,3);
\draw (4,2) rectangle (5,3);
\draw (5,2) rectangle (6,3);
\draw (6,2) rectangle (7,3);
\draw (7,2) rectangle (8,3);

\draw (0,3) rectangle (1,4);
\draw (1,3) rectangle (2,4);
\draw (2,3) rectangle (3,4);
\draw (3,3) rectangle (4,4);
\draw (4,3) rectangle (5,4);
\draw (5,3) rectangle (6,4);
\draw (6,3) rectangle (7,4);
\draw (7,3) rectangle (8,4);

\end{tikzpicture}
\end{center}
\end{minipage}

\vspace{0.25in}

\begin{minipage}{0.49\textwidth}
\begin{center}
\begin{tikzpicture}[scale=0.85]

\fill[gray] (0,0) rectangle (1,1);
\draw (0.5,0.5) node {\faBomb};

\draw (4.5,2.5) node {\small $p$};
\draw (3.5,2.5) node {\small $\frac{1-p}{4}$};
\draw (3.5,1.5) node {\small $\frac{1-p}{2}$};
\draw (4.5,1.5) node {\small $\frac{1-p}{4}$};

\draw (0,0) rectangle (1,1);
\draw (1,0) rectangle (2,1);
\draw (2,0) rectangle (3,1);
\draw (3,0) rectangle (4,1);
\draw (4,0) rectangle (5,1);
\draw (5,0) rectangle (6,1);
\draw (6,0) rectangle (7,1);
\draw (7,0) rectangle (8,1);

\draw (0,1) rectangle (1,2);
\draw (1,1) rectangle (2,2);
\draw (2,1) rectangle (3,2);
\draw (3,1) rectangle (4,2);
\draw (4,1) rectangle (5,2);
\draw (5,1) rectangle (6,2);
\draw (6,1) rectangle (7,2);
\draw (7,1) rectangle (8,2);

\draw (0,2) rectangle (1,3);
\draw (1,2) rectangle (2,3);
\draw (2,2) rectangle (3,3);
\draw (3,2) rectangle (4,3);
\draw (4,2) rectangle (5,3);
\draw (5,2) rectangle (6,3);
\draw (6,2) rectangle (7,3);
\draw (7,2) rectangle (8,3);

\draw (0,3) rectangle (1,4);
\draw (1,3) rectangle (2,4);
\draw (2,3) rectangle (3,4);
\draw (3,3) rectangle (4,4);
\draw (4,3) rectangle (5,4);
\draw (5,3) rectangle (6,4);
\draw (6,3) rectangle (7,4);
\draw (7,3) rectangle (8,4);

\end{tikzpicture}\end{center}\end{minipage}
\begin{minipage}{0.49\textwidth}
\begin{center}
\begin{tikzpicture}[scale=0.85]

\fill[gray] (0,0) rectangle (1,1);
\draw (0.5,0.5) node {\faBomb};

\draw (4.5,2.5) node {\small $p$};

\draw (0.5,2.5) node {\small $\alpha$};
\draw (1.5,2.5) node {\small $\alpha$};
\draw (2.5,2.5) node {\small $\alpha$};
\draw (3.5,2.5) node {\small $\alpha$};
\draw (5.5,2.5) node {\small $\alpha$};
\draw (6.5,2.5) node {\small $\alpha$};
\draw (7.5,2.5) node {\small $\alpha$};

\draw (4.5,0.5) node {\small $\alpha$};
\draw (4.5,1.5) node {\small $\alpha$};
\draw (4.5,3.5) node {\small $\alpha$};

\draw (0,0) grid (8,4);

\end{tikzpicture}
\end{center}
\end{minipage}
\caption{Examples of possible error probability distributions in
\emph{Chomp!}; squares with probability $0$ are left blank.
\textbf{Top left:} the error is distributed uniformly among the squares
sharing an edge or a vertex with the targeted square.
\textbf{Top right:} the error is distributed uniformly among the squares
sharing an edge with the targeted square.
\textbf{Bottom left:} an example of a non-uniform distribution among
neighboring squares.
\textbf{Bottom right:} the error is distributed uniformly among all
non-targeted squares in the same row or column as the targeted square,
where $\alpha=(1-p)/(n+m-2)$.}
\label{Figure:Chomp}
\end{figure}

For example, consider the $2\times2$ case with the error distribution
shown in the top-left diagram of Figure~\ref{Figure:Chomp}. Thus, the
targeted square is hit with probability $p$, while an error can only
lead to a neighboring square, that is, one sharing an edge or a vertex
with the targeted square. The remaining probability $1-p$ is distributed
uniformly among all such neighbors. Figure~\ref{Figure:Distribution}
illustrates how the distribution is adjusted when the targeted square
does not have the full set of neighbors. 
For simplicity of notation, we omit the dependence on $p$
when writing the winning probabilities of the positions
in the following analysis.
Clearly, the probability $N_{\tikz{\fill[gray] (0,0) rectangle (0.1,0.1);\draw (0,0) rectangle (0.1,0.1)}}$ for Player I to win from position $\tikz{\fill[gray] (0,0) rectangle (0.15,0.15);\draw (0,0) rectangle (0.15,0.15)}$ is $0$. Also, the chance for Player I to win from positions $\tikz{\fill[gray] (0,0) rectangle (0.15,0.15);\draw (0,0) rectangle (0.15,0.15);\draw (0.15,0) rectangle (0.3,0.15);}$ or $\tikz{\fill[gray] (0,0) rectangle (0.15,0.15);\draw (0,0) rectangle (0.15,0.15);\draw (0,0.15) rectangle (0.15,0.3);}$ is  $N_{\tikz{\fill[gray] (0,0) rectangle (0.1,0.1);\draw (0,0) rectangle (0.1,0.1);\draw (0.1,0) rectangle (0.2,0.1);}}=N_{\tikz{\fill[gray] (0,0) rectangle (0.1,0.1);\draw (0,0) rectangle (0.1,0.1);\draw (0,0.1) rectangle (0.1,0.2);}}=1$. Furthermore, since every valid move from $\tikz{\fill[gray] (0,0) rectangle (0.15,0.15);\draw (0,0) rectangle (0.15,0.15);\draw (0,0.15) rectangle (0.15,0.3);\draw (0.15,0) rectangle (0.3,0.15);}$ leads to a winning position for the next player, we have $N_{\tikz{\fill[gray] (0,0) rectangle (0.1,0.1);\draw (0,0) rectangle (0.1,0.1);\draw (0,0.1) rectangle (0.1,0.2);\draw (0.1,0) rectangle (0.2,0.1);}}=0$. We now compute Player I's winning probability for each of the three
possible transmitted moves from $\tikz{\fill[gray] (0,0) rectangle (0.15,0.15);\draw (0,0) rectangle (0.15,0.15);\draw (0,0.15) rectangle (0.15,0.3);\draw (0.15,0) rectangle (0.3,0.15);\draw (0.15,0.15) rectangle (0.3,0.3);}$ (the symbol $\times$ indicates the transmitted move). 

\begin{itemize}

\item $\text{Winning probability if Player I transmits $\tikz[baseline]{\fill[gray] (0,0) rectangle (0.2,0.2);\draw (0,0) rectangle (0.2,0.2);\draw (0,0.2) rectangle (0.2,0.4);\draw (0.2,0) rectangle (0.4,0.2);\draw (0.2,0.2) rectangle (0.4,0.4);\draw (0.1,0.3) node {$\mathbf{\times}$};}$ = }(1-N_{\tikz{\fill[gray] (0,0) rectangle (0.1,0.1);\draw (0,0) rectangle (0.1,0.1);\draw (0.1,0) rectangle (0.2,0.1);}})p + (1-N_{\tikz{\fill[gray] (0,0) rectangle (0.1,0.1);\draw (0,0) rectangle (0.1,0.1);\draw (0,0.1) rectangle (0.1,0.2);}})\frac{1-p}{2}+ (1-N_{\tikz{\fill[gray] (0,0) rectangle (0.1,0.1);\draw (0,0) rectangle (0.1,0.1);\draw (0,0.1) rectangle (0.1,0.2);\draw (0.1,0) rectangle (0.2,0.1);}})\frac{1-p}{2}=\frac{1-p}{2}.$

\vspace{0.2in}

\item $\text{Winning probability if Player I transmits $\tikz[baseline]{\fill[gray] (0,0) rectangle (0.2,0.2);\draw (0,0) rectangle (0.2,0.2);\draw (0,0.2) rectangle (0.2,0.4);\draw (0.2,0) rectangle (0.4,0.2);\draw (0.2,0.2) rectangle (0.4,0.4); \draw (0.3,0.1) node {$\mathbf{\times}$};}$ = }(1-N_{\tikz{\fill[gray] (0,0) rectangle (0.1,0.1);\draw (0,0) rectangle (0.1,0.1);\draw (0.1,0) rectangle (0.2,0.1);}})\frac{1-p}{2} + (1-N_{\tikz{\fill[gray] (0,0) rectangle (0.1,0.1);\draw (0,0) rectangle (0.1,0.1);\draw (0,0.1) rectangle (0.1,0.2);}})p + (1-N_{\tikz{\fill[gray] (0,0) rectangle (0.1,0.1);\draw (0,0) rectangle (0.1,0.1);\draw (0,0.1) rectangle (0.1,0.2);\draw (0.1,0) rectangle (0.2,0.1);}})\frac{1-p}{2}=\frac{1-p}{2}.$

\vspace{0.2in}

\item $\text{Winning probability if Player I transmits $\tikz{\fill[gray] (0,0) rectangle (0.2,0.2);\draw (0,0) rectangle (0.2,0.2);\draw (0,0.2) rectangle (0.2,0.4);\draw (0.2,0) rectangle (0.4,0.2);\draw (0.2,0.2) rectangle (0.4,0.4);\draw (0.3,0.3) node {$\mathbf{\times}$};}$ = }(1-N_{\tikz{\fill[gray] (0,0) rectangle (0.1,0.1);\draw (0,0) rectangle (0.1,0.1);\draw (0.1,0) rectangle (0.2,0.1);}})\frac{1-p}{2} + (1-N_{\tikz{\fill[gray] (0,0) rectangle (0.1,0.1);\draw (0,0) rectangle (0.1,0.1);\draw (0,0.1) rectangle (0.1,0.2);}})\frac{1-p}{2}+ (1-N_{\tikz{\fill[gray] (0,0) rectangle (0.1,0.1);\draw (0,0) rectangle (0.1,0.1);\draw (0,0.1) rectangle (0.1,0.2);\draw (0.1,0) rectangle (0.2,0.1);}})p=p.$

\end{itemize}

\vspace{0.1in}

\noindent Therefore, the optimal winning probability $N_{\tikz{\fill[gray] (0,0) rectangle (0.1,0.1);\draw (0,0) rectangle (0.1,0.1);\draw (0,0.1) rectangle (0.1,0.2);\draw (0.1,0) rectangle (0.2,0.1);\draw (0.1,0.1) rectangle (0.2,0.2);}}$ that Player I wins from position $\tikz{\fill[gray] (0,0) rectangle (0.15,0.15);\draw (0,0) rectangle (0.15,0.15);\draw (0,0.15) rectangle (0.15,0.3);\draw (0.15,0) rectangle (0.3,0.15);\draw (0.15,0.15) rectangle (0.3,0.3);}$ is

\begin{equation}\label{Eq:Chomp}N_{\tikz{\fill[gray] (0,0) rectangle (0.1,0.1);\draw (0,0) rectangle (0.1,0.1);\draw (0,0.1) rectangle (0.1,0.2);\draw (0.1,0) rectangle (0.2,0.1);\draw (0.1,0.1) rectangle (0.2,0.2);}}=\max\left\{p,\dfrac{1-p}{2}\right\}=\left\{\begin{array}{cll} \dfrac{1-p}{2}& 0\leq p\leq 1/3&\tikz[baseline]{\fill[gray] (0,0) rectangle (0.2,0.2);\draw (0,0) rectangle (0.2,0.2);\draw (0,0.2) rectangle (0.2,0.4);\draw (0.2,0) rectangle (0.4,0.2);\draw (0.2,0.2) rectangle (0.4,0.4); \draw (0.3,0.1) node {$\mathbf{\times}$};} \tikz{\fill[gray] (0,0) rectangle (0.2,0.2);\draw (0,0) rectangle (0.2,0.2);\draw (0,0.2) rectangle (0.2,0.4);\draw (0.2,0) rectangle (0.4,0.2);\draw (0.2,0.2) rectangle (0.4,0.4);\draw (0.1,0.3) node {$\mathbf{\times}$};}\\&&\\ p& 1/3\leq p\leq 1&\tikz{\fill[gray] (0,0) rectangle (0.2,0.2);\draw (0,0) rectangle (0.2,0.2);\draw (0,0.2) rectangle (0.2,0.4);\draw (0.2,0) rectangle (0.4,0.2);\draw (0.2,0.2) rectangle (0.4,0.4);\draw (0.3,0.3) node {$\mathbf{\times}$};}\end{array}\right.\end{equation}

\vspace{0.1in}

\noindent Notice that, for $p=\frac{1}{3}$, $N_{\tikz{\fill[gray] (0,0) rectangle (0.1,0.1);\draw (0,0) rectangle (0.1,0.1);\draw (0,0.1) rectangle (0.1,0.2);\draw (0.1,0) rectangle (0.2,0.1);\draw (0.1,0.1) rectangle (0.2,0.2);}}=\frac{1}{3}$ for every transmitted move. Thus, the game is a chance game favoring Player II.

The remaining positions
$\tikz{\fill[gray] (0,0) rectangle (0.15,0.15);\draw (0,0) rectangle (0.15,0.15);\draw (0.15,0) rectangle (0.3,0.15);\draw (0.3,0) rectangle (0.45,0.15);}$,
$\tikz{\fill[gray] (0,0) rectangle (0.15,0.15);\draw (0,0) rectangle (0.15,0.15);\draw (0,0.15) rectangle (0.15,0.3);\draw (0.15,0) rectangle (0.3,0.15);\draw (0.3,0) rectangle (0.45,0.15);}$
and
$\tikz{\fill[gray] (0,0) rectangle (0.15,0.15);\draw (0,0) rectangle (0.15,0.15);\draw (0,0.15) rectangle (0.15,0.3);\draw (0.15,0) rectangle (0.3,0.15);\draw (0.15,0.15) rectangle (0.3,0.3);\draw (0.3,0) rectangle (0.45,0.15);}$
can be analyzed in the same way to solve the $2\times3$ case. The probability
distributions associated with the possible transmitted moves from
$\tikz{\fill[gray] (0,0) rectangle (0.15,0.15);\draw (0,0) rectangle (0.15,0.15);\draw (0,0.15) rectangle (0.15,0.3);\draw (0.15,0) rectangle (0.3,0.15);\draw (0.15,0.15) rectangle (0.3,0.3);\draw (0.3,0) rectangle (0.45,0.15);}$
are shown in Figure~\ref{Figure:Distribution}.

\begin{figure}
\begin{center}
\begin{tikzpicture}


\fill[gray] (-4,0) rectangle (-3,1);
\draw (-3.5,0.5) node {\faBomb};

\draw (-2.5,1.5) node {$\frac{1-p}{2}$};
\draw (-1.5,0.5) node {$p$};
\draw (-2.5,0.5) node {$\frac{1-p}{2}$};
\draw (-3.5,1.5) node {$0$};

\draw (-4,0) rectangle (-3,1);
\draw (-3,0) rectangle (-2,1);
\draw (-4,1) rectangle (-3,2);
\draw (-3,1) rectangle (-2,2);
\draw (-2,0) rectangle (-1,1);

\fill[gray] (0,0) rectangle (1,1);
\draw (0.5,0.5) node {\faBomb};

\draw (1.5,1.5) node {$\frac{1-p}{3}$};
\draw (2.5,0.5) node {$\frac{1-p}{3}$};
\draw (0.5,1.5) node {$\frac{1-p}{3}$};
\draw (1.5,0.5) node {$p$};

\draw (0,0) rectangle (1,1);
\draw (1,0) rectangle (2,1);
\draw (0,1) rectangle (1,2);
\draw (1,1) rectangle (2,2);
\draw (2,0) rectangle (3,1);

\fill[gray] (4,0) rectangle (5,1);
\draw (4.5,0.5) node {\faBomb};

\draw (5.5,1.5) node {$p$};
\draw (4.5,1.5) node {$\frac{1-p}{3}$};
\draw (5.5,0.5) node {$\frac{1-p}{3}$};
\draw (6.5,0.5) node {$\frac{1-p}{3}$};

\draw (4,0) rectangle (5,1);
\draw (5,0) rectangle (6,1);
\draw (4,1) rectangle (5,2);
\draw (5,1) rectangle (6,2);
\draw (6,0) rectangle (7,1);

\fill[gray] (8,0) rectangle (9,1);
\draw (8.5,0.5) node {\faBomb};

\draw (9.5,1.5) node {$\frac{1-p}{2}$};
\draw (8.5,1.5) node {$p$};
\draw (9.5,0.5) node {$\frac{1-p}{2}$};
\draw (10.5,0.5) node {$0$};

\draw (8,0) rectangle (9,1);
\draw (9,0) rectangle (10,1);
\draw (8,1) rectangle (9,2);
\draw (9,1) rectangle (10,2);
\draw (10,0) rectangle (11,1);

\end{tikzpicture}
\end{center}
\caption{Probability distributions associated with the possible
transmitted moves from the displayed position.}
\label{Figure:Distribution}
\end{figure}

\section{A general framework}\label{Sec:Framework}

Throughout this section, let $G=(V,E)$ be a progressively bounded
directed graph. Recall that, for $v\in V$, the set of followers of $v$ is
\(
F(v)=\{w\in V:(v,w)\in E\}.
\)
For $e\in E$, let $t(e)\in V$ denote the tail of $e$. If $e,e'\in E$
satisfy $t(e)=t(e')$, we write $e\sim e'$. For $v\in V$, we also set
\(
E_v=\{e\in E:t(e)=v\}.
\)
A standard reference for graph theory is \cite{West} (see also \cite[\S 1]{Ferguson}).

\begin{defi} Let $G=(V,E)$ be a combinatorial game. A \emph{move error function} $\psi$ over $G$ is a map $\psi: E\times E\rightarrow [0,1]$ satisfying:\begin{enumerate}
\item[($i$)] $\psi(e,e')=0$ if $e\nsim e'$.
\item[($ii$)] For every fixed $e\in E$, $\displaystyle\sum_{e'\sim e}\psi(e,e')=1$.
\end{enumerate}
Since $\psi$ can take non-zero values only on pairs of edges with the
same tail, for $v\in V$ and $w,u\in F(v)$ we shall write
\[
\psi_v(w,u)=\psi((v,w),(v,u)).
\]\end{defi}

\noindent The function $\psi(e,e')$ represents the probability that the move $e'$ is
effectively played when the move $e$ is transmitted, provided that
$t(e)=t(e')$.

\begin{obs} A move error function on $G$ is equivalently given by a collection
$\{\psi_v\}_{v\in V}$ of maps
$\psi_v:F(v)\times F(v)\longrightarrow[0,1]$
such that
\(
\sum_{u\in F(v)}\psi_v(w,u)=1
\)
for every $v\in V$ and every $w\in F(v)$.
If the edges of $G$ are ordered, say $E=\{e_1,\ldots,e_r\}$, then a
move error function $\psi$ can also be represented by the stochastic
matrix $D\in\mathbb{R}^{r\times r}$ defined by
$D_{ij}=\psi(e_i,e_j)$. Conversely, a stochastic matrix
$D\in\mathbb{R}^{r\times r}$ determines a move error function on $G$
whenever $D_{ij}=0$ for all pairs $e_i\not\sim e_j$.\end{obs}

\vspace{0.1in}

Given a digraph $G=(V,E)$ and a move error function $\psi$ on $G$, we
define the corresponding two-player win-loss game $G_\psi$ by the
following rules:
\begin{enumerate}
\item Player I moves first from the starting position of $G$.
\item The players alternate turns.
\item At a position $v\in V$, the player whose turn it is chooses a
follower $w\in F(v)$, that is, transmits the move $v\to w$. The move
that is effectively played is then $v\to u$, where $u\in F(v)$ is chosen
randomly according to the probabilities $\psi_v(w,u)$.
\item A player whose turn begins at a terminal position loses.
\end{enumerate}

\begin{ej}\label{Ej:NoError}
For every $v\in V$, define
\[
\psi_v(w,u)=
\begin{cases}
1, & w=u,\\
0, & w\neq u.
\end{cases}
\]
Then $G_\psi=G$. Thus, ordinary combinatorial games are a particular
case of the present framework.
\end{ej}

\begin{ej}\label{Ej:Nim} Consider 1-pile \emph{Nim} with $k$ chips and the error mechanism
described in Section~\ref{Sec:MotivatingExamples}. If $s$ is the transmitted move
and $r$ is the move that is effectively played, then
\[
\psi_k(s,r)=
\frac{
p^{d_H(s,r)}(1-p)^{n-d_H(s,r)}
}{
\displaystyle\sum_{i=0}^{k-1}
p^{d_H(s,i)}(1-p)^{n-d_H(s,i)}
},
\]
where $p$ is the probability of changing a single transmitted digit,
$d_H(s,r)$ is the Hamming distance between the binary representations
of $s$ and $r$, and $n$ is the number of digits in the binary
representation of $k$.\end{ej}

\begin{ej} Let $G$ be the digraph of 1-pile \emph{Nim} with more than one chip.
Write $v_i$ for the vertex representing the position with $i$ chips,
$0\leq i\leq n$, and define
\[
\psi_{v_i}(w,u)=\frac{1}{i}
\]
for every $w,u\in F(v_i)$ and every $i\neq0$. Thus, from each
non-terminal position, the move that is effectively played is chosen
uniformly among all legal followers, independently of the transmitted
move. Then $G_\psi$ is a fair chance game (see Lemma \ref{Lemma:MakingChance}).
\end{ej}

Consider a game $G_\psi$. For each non-terminal position $v\in V$ and
each $w\in F(v)$, let $N_v(w)$ denote the probability that the next
player wins from $v$ when the move $v\to w$ is transmitted. We define
the optimal winning probability from $v$ by
\[
N_v=
\begin{cases}
\displaystyle\max_{w\in F(v)}N_v(w),
& \text{if $v$ is non-terminal},\\[2mm]
0,
& \text{if $v$ is terminal}.
\end{cases}
\]

\noindent Solving the game $G_\psi$  amounts to determining, for every
position $v\in V$, \begin{enumerate}
\item the value $N_v$;
\item the optimal transmitted move or moves from $v$.
\end{enumerate}

As illustrated by the examples in
Section~\ref{Sec:MotivatingExamples}, the winning
probabilities can be computed recursively. Let $v$ be a non-terminal position and fix a transmitted
move $v\to w$, with $w\in F(v)$. Due to the error mechanism,
the move effectively played may be $v\to u$, for any
$u\in F(v)$, with probability $\psi_v(w,u)$. Conditioning on the position actually reached, we obtain
\begin{align}
N_v(w)
&=\sum_{u\in F(v)}
P(\text{current player wins}\mid
\text{position }u\text{ is reached})
\,\psi_v(w,u) \nonumber\\
&=\sum_{u\in F(v)}
(1-N_u)\,\psi_v(w,u).
\label{Eq:mainformula}
\end{align}

\noindent Indeed, once position $u$ is reached, the opponent
plays optimally from $u$ and wins with probability
$N_u$. Therefore, the probability that the current
player eventually wins is $1-N_u$.

Now, since the player chooses the transmitted move
that maximizes the probability of winning, we have
\[
N_v=\max_{w\in F(v)}
\sum_{u\in F(v)}(1-N_u)\,\psi_v(w,u).
\]

\noindent Together with $N_v=0$ for terminal positions,
this gives a recursive procedure for determining
the winning probabilities and the corresponding
optimal transmitted moves.

\begin{obs} Equation~\eqref{Eq:mainformula} gives simple characterizations of the
positions with winning probability $0$ or $1$. A position $v$ satisfies $N_v=0$ if and only if, for every transmitted
move $w\in F(v)$ and every $u\in F(v)$ such that
$\psi_v(w,u)>0$, one has $N_u=1$. Indeed, all the terms in
Equation~\eqref{Eq:mainformula} are non-negative, and $N_v=0$ means
that $N_v(w)=0$ for every $w\in F(v)$. Similarly, $N_v=1$ if and only if there exists a transmitted move
$w\in F(v)$ such that $N_u=0$ for every $u\in F(v)$ with
$\psi_v(w,u)>0$. In particular, if $\psi_v(w,u)>0$ for every $w,u\in F(v)$, then
\(
N_v=0 \Leftrightarrow N_u=1\text{ } 
\forall u\in F(v)
\text{ and }
N_v=1 \Leftrightarrow N_u=0
\text{ } 
\forall u\in F(v).
\)\end{obs}

 For a finite game \(G_{\psi}\), every move error function can be approximated arbitrarily well by one which takes strictly positive values on every pair of arrows sharing a tail. Moreover, if every non-terminal position has a unique optimal move, the approximation can be chosen without changing any of these optimal moves. This shows that winning probabilities and, under the uniqueness assumption, optimal moves are stable under small perturbations of the error mechanism, as established in the following proposition.

\begin{prop}
Let $G_\psi$ be a finite game. Given $\varepsilon>0$, there exists a
move error function $\psi'$ on $G$ such that:
\begin{enumerate}
    \item $\psi'(e,e')>0$ for every $e\sim e'$;
    \item $|N_v-N'_v|<\varepsilon$ for every $v\in V$, where $N'_v$
    denotes the winning probability from $v$ in $G_{\psi'}$.
\end{enumerate}
Moreover, if every non-terminal position has a unique optimal move in
$G_\psi$, then $\psi'$ can be chosen so that the optimal move from
every position is the same in $G_\psi$ and $G_{\psi'}$.
\end{prop}

\begin{proof}
For $0<\delta<1$ and every non-terminal position $v$, define
\[
\psi^\delta_v(w,u)
=
(1-\delta)\psi_v(w,u)
+\frac{\delta}{|F(v)|},
\qquad w,u\in F(v).
\]
For every fixed $w\in F(v)$ we have
\[
\sum_{u\in F(v)}\psi^\delta_v(w,u)
=
(1-\delta)\sum_{u\in F(v)}\psi_v(w,u)+\delta
=1 \quad \text{ and }\quad \psi^\delta_v(w,u)\geq\frac{\delta}{|F(v)|}>0.
\]
Hence $\psi^\delta$ is a move error function satisfying condition~(1).

Let $N^\delta_v$ denote the winning probability from $v$ in
$G_{\psi^\delta}$. We claim that $N^\delta_v\to N_v$ as
$\delta\to0$ for every $v\in V$. We prove this by backward induction
on the positions of the game. If $v$ is terminal, then
$N^\delta_v=N_v=0$. Suppose now that
$N^\delta_u\to N_u$ for every $u\in F(v)$. For each $w\in F(v)$,
Equation~\eqref{Eq:mainformula} gives
\[
N^\delta_v(w)
=
\sum_{u\in F(v)}
(1-N^\delta_u)\psi^\delta_v(w,u).
\]
By the inductive hypothesis and the fact that
$\psi^\delta_v(w,u)\to\psi_v(w,u)$, we obtain
$N^\delta_v(w)\to N_v(w)$ for every $w\in F(v)$. Since $F(v)$ is
finite and
\(
N^\delta_v=\max_{w\in F(v)}N^\delta_v(w),
\)
it follows that $N^\delta_v\to N_v$. Since $V$ is finite, this convergence is uniform over all positions.
Therefore, for $\delta>0$ sufficiently small,
\(
|N^\delta_v-N_v|<\varepsilon
\)
for every $v\in V$. Taking $\psi'=\psi^\delta$ proves condition~(2).

Finally, suppose that every non-terminal position has a unique optimal
move. For each position $v$ having at least two followers, let $w_v$
denote its unique optimal move and define
\[
\gamma_v
=
N_v(w_v)
-
\max_{\substack{w\in F(v)\\ w\neq w_v}}N_v(w)>0.
\]
If there are no such positions, the conclusion is immediate. Otherwise,
since the game is finite, we may set
\(
\gamma=\min_v\gamma_v>0,
\)
where the minimum is taken over all positions having at least two
followers. The preceding argument also shows that
$N^\delta_v(w)\to N_v(w)$ for every $v\in V$ and every $w\in F(v)$.
Since there are only finitely many such pairs, by taking $\delta$
smaller if necessary we may assume that
\(
|N^\delta_v(w)-N_v(w)|<\frac{\gamma}{3}
\)
for all of them. Hence, if $w\neq w_v$,
$$N^\delta_v(w_v)-N^\delta_v(w)
>
N_v(w_v)-N_v(w)-\frac{2\gamma}{3}\geq\frac{\gamma}{3}>0$$
Thus $w_v$ remains the unique optimal move from $v$ in
$G_{\psi^\delta}$. Positions with only one follower are immediate.
Therefore, for $\delta$ sufficiently small, the optimal moves in
$G_\psi$ and $G_{\psi^\delta}$ coincide. Taking again
$\psi'=\psi^\delta$ completes the proof.
\end{proof}

In analogy with the classical theory, we call a position $v\in G_\psi$
a $P$-position if $N_v=0$, and an $N$-position if $N_v=1$.
A position $v$ such that $0<N_v<1$ will be called an $O$-position.

A suitable move error function may turn a combinatorial game into a
chance game. As observed in Section~\ref{Sec:MotivatingExamples}, 1-pile \emph{Nim}
with three chips is a fair chance game when $p=\frac{1}{2}$, whereas the
$2\times2$ example of \emph{Chomp!} considered there is an unfair chance
game when $p=1/3$. The following result gives a sufficient condition
under which a uniform move error function produces a fair chance game. 

\begin{lema}\label{Lemma:MakingChance} Suppose that, for every non-terminal position $v\in V$,
\[
\psi_v(w,u)=\frac{1}{|F(v)|}
\]
for all $w,u\in F(v)$. Assume moreover that
\[
|\{u\in F(v):N_u=0\}|
=
|\{u\in F(v):N_u=1\}|
\]
whenever $0<N_v<1$. Then every $O$-position satisfies
$N_v=\frac{1}{2}$. In particular, if the starting position is an $O$-position,
then $G_\psi$ is a fair chance game.
\end{lema}

\begin{proof}
We argue by induction on the positions of the game. Let $v$ be an
$O$-position and assume that every $O$-position among its followers has
winning probability $\frac{1}{2}$. Since the move error function is uniform,
Equation~\eqref{Eq:mainformula} gives, independently of the transmitted
move $w$,
\[
N_v(w)
=
\frac{1}{|F(v)|}
\sum_{u\in F(v)}(1-N_u).
\]
Partition $F(v)$ according to whether $N_u=0$, $N_u=1$, or
$0<N_u<1$. By the inductive hypothesis, the followers in the last class
have value $\frac{1}{2}$. Since the numbers of $P$- and $N$-followers are equal,
their total contributions to the preceding sum are also balanced.
Therefore,
\[
N_v(w)
=
\frac{1}{|F(v)|}
\left(
|\{u:N_u=0\}|
+\frac12|\{u:0<N_u<1\}|
\right)=\frac12
\]
Hence $N_v=\frac{1}{2}$.
\end{proof}

We now verify that multi-pile \emph{Nim}, equipped with the uniform move error function, satisfies the hypotheses of Lemma \ref{Lemma:MakingChance}. Among the positions consisting only of piles of size 0 or 1, the $N$-positions are those with an odd number of non-empty piles, whereas the $P$-positions are those with an even number of non-empty piles. Moreover, every position containing at least one pile with two or more chips is an $O$-position, as can be shown by induction on the total number of chips. Indeed, if exactly one pile contains at least two chips, reducing it to 0 or 1 produces a $P$-follower and an $N$-follower. Since both are reached with positive probability, the winning probability lies strictly between 0 and 1. If at least two piles contain two or more chips, every follower still contains such a pile and is therefore an $O$-position by the induction hypothesis. Hence the original position is also an $O$-position.

Now consider an $O$-position \((n_1,\ldots,n_k)\). Such a position can have a $P$- or $N$-follower only when exactly one of its piles contains at least two chips, that is, when it has the form \(
(*,\ldots,*,n,*,\ldots,*), n\geq 2,
\)
where every entry denoted by \(*\) is either 0 or 1. In this case, the only $P$- and $N$-followers are obtained by replacing \(n\) by 0 or 1, one of each. If at least two piles contain two or more chips, there are no $P$- or $N$-followers. Thus every $O$-position has the same number of $P$- and $N$-followers, and Lemma ~\ref{Lemma:MakingChance} applies.

\begin{coro}\label{Coro:NimBalanced} Let $G$ be the digraph of multi-pile \emph{Nim}, and let $\psi$ be the
uniform move error function. Then every $O$-position of $G_\psi$ has
winning probability $\frac{1}{2}$. In particular, if the starting position is an
$O$-position, then $G_\psi$ is a fair chance game.\end{coro}
On the other hand, \emph{Chomp!} does not necessarily satisfy the balancing condition of 
Lemma~\ref{Lemma:MakingChance}. Indeed, the following position provides a counterexample:
 $\tikz{\fill[gray] (0,0) rectangle (0.15,0.15);\draw (0,0) rectangle (0.15,0.15);\draw (0,0.15) rectangle (0.15,0.3);\draw (0.15,0) rectangle (0.3,0.15);\draw (0.15,0.15) rectangle (0.3,0.3)}$

\section{1-Pile Nim Revisited}\label{Sect:ToyExampleRevisited}

In this section we return to the case of 1-pile \textit{Nim} with $k$
chips introduced in Section~\ref{Sec:MotivatingExamples}. This example turns out to
have a non-trivial behavior for general $k$. Recall that the number of
chips in a given position is encoded by its binary representation, and a
move is transmitted by specifying the number of chips left in the heap
after that move. As before, we let $N_k(p)$ denote the probability of winning from a
heap of size $k$, where $p$ is the probability of making an error in a
single digit. Throughout this section we assume $0\leq p<1$. The case $p=1$ requires
separate consideration, since for some transmitted words every received
word may be invalid, so that the retransmission procedure need not
terminate. The limiting behavior as $p\to1^-$ will be studied in
Subsection~\ref{Subsec:HighError}.

Equation~\eqref{Eq:mainformula} and Example~\ref{Ej:Nim} give
\begin{equation}\label{Eq:formularNim}
N_k(p)=\max_{0\leq s\leq k-1}
\sum_{i=0}^{k-1} (1-N_i(p))
\dfrac{p^{d_H(s,i)}(1-p)^{n-d_H(s,i)}}
{\displaystyle\sum_{j=0}^{k-1}
p^{d_H(s,j)}(1-p)^{n-d_H(s,j)}}.
\end{equation}

\noindent Here $d_H(s,r)$ is the Hamming distance between the binary
representations of $s$ and $r$, and $n$ is the number of digits in the
binary representation of $k$. All binary representations are taken to have length \(n\), with leading zeros added when necessary.  Clearly, $N_0(p)=0$ and $N_1(p)=1$ for
every $0\leq p<1$. Also, it is straightforward to see that
\[
N_2(p)=
\begin{cases}
1-p, & 0\leq p\leq\frac{1}{2},\\
p,   & \frac{1}{2}\leq p<1.
\end{cases}
\]

\noindent We already found $N_3(p)$ in Section~\ref{Sec:MotivatingExamples}
(see Equation~\eqref{Eq:Nim}). For larger heaps, the behavior of the
game depends strongly on whether the error probability is below or
above $\frac{1}{2}$. We therefore divide the analysis into two regimes. We
first show that, for $0\leq p\leq\frac{1}{2}$, the optimal strategy has a
remarkably simple form and obtain an exact expression for the winning
probabilities. We then turn to the case $p>\frac{1}{2}$, where the optimal
strategy exhibits a substantially richer behavior.

\subsection{The case $0\leq p\leq \frac{1}{2}$}\label{Subsec:LowError}

For $0<p\leq\frac{1}{2}$, set
\[
q=\frac{p}{1-p},
\]
so that $0<q\leq1$. Throughout this subsection we omit $p$ from the
notation when it is fixed. Consistently with the notation introduced in
Section~\ref{Sec:Framework}, for a transmitted move $s$ we write
\[
N_k(s)=
\frac{\displaystyle\sum_{i=0}^{k-1}(1-N_i)q^{d_H(s,i)}}
{\displaystyle\sum_{i=0}^{k-1}q^{d_H(s,i)}}.
\]
Thus
\(
N_k=\max_{0\leq s\leq k-1}N_k(s).
\)
Let $\operatorname{wt}(i)$ denote the Hamming weight of the binary
representation of $i$. Since $d_H(0,i)=\operatorname{wt}(i)$,
transmitting $0$ gives
\[
N_k(0)=
\frac{\displaystyle\sum_{i=0}^{k-1}(1-N_i)q^{\operatorname{wt}(i)}}
{\displaystyle\sum_{i=0}^{k-1}q^{\operatorname{wt}(i)}}.
\]
The following result determines the optimal move in the low-error
regime.

\begin{teo}\label{Teo:LowError}
Let $k\geq1$ and $0\leq p\leq\frac{1}{2}$. Then transmitting $0$ is an
optimal move and
\[
N_k(p)\geq\frac12.
\]
If $0\leq p<\frac{1}{2}$, then $N_k(p)>\frac{1}{2}$ and transmitting $0$ is the
unique optimal move. If $p=\frac{1}{2}$ and $k\geq2$, then
$N_k(p)=\frac{1}{2}$ and every move is optimal.
\end{teo}

\begin{proof}
For $p=0$, there are no transmission errors, so
transmitting $0$ is clearly the unique optimal move
and $N_k(0)=1$. We now assume $0<p\leq 1/2$ and proceed by
induction on $k$. The
case $k=1$ is immediate, since the only possible move is $0$ and
$N_1=1$.

Suppose that the result holds for every heap of size $i<k$. Thus,
transmitting $0$ is optimal and $N_i\geq\frac{1}{2}$ for every
$1\leq i<k$. Moreover, if $0<p<\frac{1}{2}$, then $N_i>\frac{1}{2}$ and
transmitting $0$ is the unique optimal move for every $1\leq i<k$.

We first consider the value of transmitting $0$. We have
\[
N_k(0)-\frac12
=
-\frac{\displaystyle
\sum_{i=0}^{k-1}
\left(N_i-\frac12\right)q^{\operatorname{wt}(i)}
}
{\displaystyle
\sum_{i=0}^{k-1}q^{\operatorname{wt}(i)}
}.
\]
By the inductive hypothesis, transmitting $0$ is optimal from a heap
of size $k-1$, and hence
\[
N_{k-1}-\frac12
=
-\frac{\displaystyle
\sum_{i=0}^{k-2}
\left(N_i-\frac12\right)q^{\operatorname{wt}(i)}
}
{\displaystyle
\sum_{i=0}^{k-2}q^{\operatorname{wt}(i)}
}.
\]
Therefore,
\begin{align*}
\sum_{i=0}^{k-1}
\left(N_i-\frac12\right)q^{\operatorname{wt}(i)}
&=
-\left(N_{k-1}-\frac12\right)
\sum_{i=0}^{k-2}q^{\operatorname{wt}(i)}
+
\left(N_{k-1}-\frac12\right)
q^{\operatorname{wt}(k-1)}
\\
&=
-\left(N_{k-1}-\frac12\right)
\left(
\sum_{i=0}^{k-2}q^{\operatorname{wt}(i)}
-
q^{\operatorname{wt}(k-1)}
\right).
\end{align*}
Now $\sum_{i=0}^{k-2}q^{\operatorname{wt}(i)}\geq1$, while
$q^{\operatorname{wt}(k-1)}\leq1$. Since $N_{k-1}\geq\frac{1}{2}$, it follows
that
\[
\sum_{i=0}^{k-1}
\left(N_i-\frac12\right)q^{\operatorname{wt}(i)}
\leq0,
\]
and consequently $N_k(0)\geq\frac{1}{2}$.

If $0<p<\frac{1}{2}$, then $0<q<1$. Since $k\geq2$,
$\operatorname{wt}(k-1)>0$, and therefore
$q^{\operatorname{wt}(k-1)}<1$. Moreover, by the inductive hypothesis,
$N_{k-1}>\frac{1}{2}$. Thus the preceding inequality is strict, and hence
$N_k(0)>\frac{1}{2}$.

We now prove that transmitting $0$ is optimal. Fix
$s\in\{0,\ldots,k-1\}$. The identity
\(
d_H(s,i)=\operatorname{wt}(s)+\operatorname{wt}(i)
-2\operatorname{wt}(s\mathbin{\&}i),
\)
where $\mathbin{\&}$ denotes the bitwise AND operation, gives
\[
q^{d_H(s,i)}
\geq
q^{\operatorname{wt}(s)}q^{\operatorname{wt}(i)},
\]
because $q\leq1$. For $i=0$, equality holds. On the other hand, the definition of $N_k(0)$ gives
\[
\sum_{i=0}^{k-1}
q^{\operatorname{wt}(i)}
\bigl(N_i+N_k(0)-1\bigr)=0.
\]
For every $i\geq1$, we have
$N_i+N_k(0)-1\geq0$. Notice that the equality for \(i=0\) is essential here, since the coefficient \(N_0+N_k(0)-1\) need not be non-negative. Therefore, comparing the terms in the two sums, we obtain
\[
\sum_{i=0}^{k-1}
q^{d_H(s,i)}
\bigl(N_i+N_k(0)-1\bigr)
\geq0.
\]
Equivalently,
\[
\frac{\displaystyle
\sum_{i=0}^{k-1}(1-N_i)q^{d_H(s,i)}
}
{\displaystyle
\sum_{i=0}^{k-1}q^{d_H(s,i)}
}
\leq N_k(0).
\]
Thus $N_k(s)\leq N_k(0)$ for every $s$, so transmitting $0$ is
optimal and $N_k=N_k(0)\geq\frac{1}{2}$.

It remains to prove uniqueness when $0<p<\frac{1}{2}$. Let
$s\in\{1,\ldots,k-1\}$. By the formula for $N_k(s)$,
$N_k(s)<N_k(0)$ is equivalent to
\[
\sum_{i=0}^{k-1}
q^{d_H(s,i)}
\bigl(N_i+N_k(0)-1\bigr)>0.
\]
We already know that
\[
q^{\operatorname{wt}(s)}
\sum_{i=0}^{k-1}
q^{\operatorname{wt}(i)}
\bigl(N_i+N_k(0)-1\bigr)=0.
\]
We compare the two sums term by term. For $i=0$ the corresponding
terms are equal, since
$d_H(s,0)=\operatorname{wt}(s)$ and $\operatorname{wt}(0)=0$.
For $1\leq i<k$, we have
\[
q^{d_H(s,i)}
\geq
q^{\operatorname{wt}(s)}q^{\operatorname{wt}(i)}
\]
and, by the inductive hypothesis together with $N_k(0)>\frac{1}{2}$,
$N_i+N_k(0)-1>0$. Finally, for $i=s$ the inequality between the
weights is strict, since
\[
q^{d_H(s,s)}=1>q^{2\operatorname{wt}(s)}.
\]
Therefore,
\[
\sum_{i=0}^{k-1}
q^{d_H(s,i)}
\bigl(N_i+N_k(0)-1\bigr)
>
q^{\operatorname{wt}(s)}
\sum_{i=0}^{k-1}
q^{\operatorname{wt}(i)}
\bigl(N_i+N_k(0)-1\bigr)
=0.
\]
Hence $N_k(s)<N_k(0)$ for every $s\neq0$, and transmitting $0$ is the
unique optimal move.

If $p=\frac{1}{2}$, then $q=1$, so all the weights $q^{d_H(s,i)}$ are equal,
independently of $s$. Hence every transmitted move has the same value.
For $k\geq2$, the inductive hypothesis gives
\[
N_k
=
\frac1k\sum_{i=0}^{k-1}(1-N_i)
=
\frac12.
\]
Thus every move is optimal.
\end{proof}

The proof of Theorem~\ref{Teo:LowError} also yields an exact recurrence
for the winning probabilities.

\begin{prop}\label{Prop:ExactLowError}
Let $0\leq p\leq\frac{1}{2}$, set $q=p/(1-p)$, and let
\[
S_j=\sum_{i=0}^{j-1}q^{\operatorname{wt}(i)}.
\]
Then, for every $k\geq2$,
\[
N_k(p)-\frac12
=
\left(N_{k-1}(p)-\frac12\right)
\frac{S_{k-1}-q^{\operatorname{wt}(k-1)}}
     {S_{k-1}+q^{\operatorname{wt}(k-1)}}.
\]
Consequently,
\[
N_k(p)
=
\frac12+
\frac12
\prod_{j=1}^{k-1}
\frac{S_j-q^{\operatorname{wt}(j)}}
     {S_j+q^{\operatorname{wt}(j)}}.
\]
\end{prop}

\begin{proof}
Assume first that $0<p\leq\frac{1}{2}$. In the proof of
Theorem~\ref{Teo:LowError} we obtained
\[
\sum_{i=0}^{k-1}
\left(N_i-\frac12\right)q^{\operatorname{wt}(i)}
=
-\left(N_{k-1}-\frac12\right)
\left(S_{k-1}-q^{\operatorname{wt}(k-1)}\right).
\]
Since transmitting $0$ is optimal, we also have
\[
N_k-\frac12
=
-\frac{\displaystyle
\sum_{i=0}^{k-1}
\left(N_i-\frac12\right)q^{\operatorname{wt}(i)}
}{S_k}.
\]
Combining these identities and using
$S_k=S_{k-1}+q^{\operatorname{wt}(k-1)}$ gives
\[
N_k-\frac12
=
\left(N_{k-1}-\frac12\right)
\frac{S_{k-1}-q^{\operatorname{wt}(k-1)}}
     {S_{k-1}+q^{\operatorname{wt}(k-1)}}.
\]
Iterating this recurrence from $N_1-\frac{1}{2}=\frac{1}{2}$ yields the product
formula.

If \(p=0\), then \(q=0\) and the winning probability is \(N_k(p)=1\) for every \(k\geq1\), so both formulas follow directly.
\end{proof}

The exact formula gives two further consequences. First, for $0<p<\frac{1}{2}$ the advantage of the player who moves first decreases
strictly with the size of the heap.

\begin{coro}\label{Coro:LowErrorMonotonicity}
For every fixed $0<p<\frac{1}{2}$,
\[
N_1(p)>N_2(p)>N_3(p)>\cdots>\frac12.
\]
\end{coro}

\begin{proof}
For $0<q<1$ and $j\geq1$, we have
$0<q^{\operatorname{wt}(j)}<1\leq S_j$. Hence
\[
0<
\frac{S_j-q^{\operatorname{wt}(j)}}
     {S_j+q^{\operatorname{wt}(j)}}
<1.
\]
Since $N_j(p)>\frac{1}{2}$ by Theorem \ref{Teo:LowError},  the recurrence in
Proposition~\ref{Prop:ExactLowError} implies
$N_{j+1}(p)<N_j(p)$ for every $j\geq1$.
\end{proof}

Secondly, the preceding monotonic sequence always converges to the value of a
fair game.

\begin{teo}\label{Teo:LowErrorLimit}
For every fixed $0<p<\frac{1}{2}$,
\[
\lim_{k\to\infty}N_k(p)=\frac12.
\]
\end{teo}

\begin{proof}
Fix $q\in(0,1)$. By Proposition~\ref{Prop:ExactLowError},
\[
N_k(p)-\frac12
=
\frac12
\prod_{j=1}^{k-1}
\frac{S_j-q^{\operatorname{wt}(j)}}
     {S_j+q^{\operatorname{wt}(j)}}.
\]
It is therefore enough to show that this product tends to $0$. Since $S_j\geq q^{\operatorname{wt}(j)}$, we have
\[
1-
\frac{S_j-q^{\operatorname{wt}(j)}}
     {S_j+q^{\operatorname{wt}(j)}}
=
\frac{2q^{\operatorname{wt}(j)}}
     {S_j+q^{\operatorname{wt}(j)}}
\geq
\frac{q^{\operatorname{wt}(j)}}{S_j}.
\]
For $m\geq1$, consider the integers
$2^{m-1}\leq j\leq2^m-1$. For each such $j$, we have
$S_j\leq S_{2^m}$. Moreover, the integers from $0$ to $2^m-1$
represent all binary words of length $m$, and hence
\[
S_{2^m}
=
\sum_{j=0}^{2^m-1}q^{\operatorname{wt}(j)}
=
(1+q)^m.
\]
The integers from $2^{m-1}$ to $2^m-1$ are precisely the binary words
of length $m$ whose first digit is $1$, so
\[
\sum_{j=2^{m-1}}^{2^m-1}q^{\operatorname{wt}(j)}
=
q(1+q)^{m-1}.
\]
Consequently,
\begin{align*}
\sum_{j=2^{m-1}}^{2^m-1}
\left(
1-
\frac{S_j-q^{\operatorname{wt}(j)}}
     {S_j+q^{\operatorname{wt}(j)}}
\right)
&\geq
\sum_{j=2^{m-1}}^{2^m-1}
\frac{q^{\operatorname{wt}(j)}}{S_j}\\
&\geq
\frac{\displaystyle
\sum_{j=2^{m-1}}^{2^m-1}q^{\operatorname{wt}(j)}
}{S_{2^m}}
=
\frac{q}{1+q}.
\end{align*}
Thus the sum, over all $j\geq1$, of one minus the corresponding
factor in the product diverges.

Finally, since each factor lies strictly between $0$ and $1$ and
$\log x\leq-(1-x)$ for $0<x<1$, the sum of the logarithms of the
factors tends to $-\infty$. Hence
\[
\prod_{j=1}^{\infty}
\frac{S_j-q^{\operatorname{wt}(j)}}
     {S_j+q^{\operatorname{wt}(j)}}
=0,
\]
and the result follows.
\end{proof}

Thus, when the probability of changing a transmitted digit is $0<p<\frac{1}{2}$, the game has a remarkably rigid behavior: attempting to remove
the entire pile is always the unique optimal strategy, while the
first-player advantage decreases with the size of the heap and
vanishes asymptotically. As we shall see next, none of this rigidity
persists in general once $p>\frac{1}{2}$.

\subsection{The case $p>\frac{1}{2}$}\label{Subsec:HighError}

The preceding results give a complete description of the game for
$0\leq p\leq \frac{1}{2}$. The situation changes substantially when
$p>\frac{1}{2}$. In particular, the argument used above no longer applies:
if $q=p/(1-p)$, then $q>1$, and the comparison between the Hamming
weights that was used in the proof is reversed. Moreover, the values
$N_i(p)-\frac{1}{2}$ need not have the same sign.

For clarity, let $N_k(p,s)$ denote the probability of winning from a
heap of size $k$ when the transmitted move is $s$, so that
$N_k(p)=\max_{0\leq s<k}N_k(p,s)$. We computed these quantities
recursively for $4\leq k\leq 10$. The approximate ranges of optimal
moves for $\frac{1}{2}<p<1$ are shown in Table~\ref{tab:higherror}. The
transition values were computed numerically; at each transition point
the two adjacent moves displayed in the table have the same winning
probability.

\begin{table}[h]
\centering
\begin{tabular}{|c|c|c|}
\hline
$k$ & range of $p$ (approx.) & optimal move $s$\\
\hline
  &                     & \\
4  & $\frac{1}{2}<p<1$                    & $3$\\
  &                     & \\
\hline
  & $\frac{1}{2}<p<0.762336$             & $3$\\
5  &                & \\
  & $0.762336<p<1$               & $4$\\
\hline
  & $\frac{1}{2}<p<0.758539$             & $3$\\
6  & $0.758539<p<0.788247$        & $0$\\
  & $0.788247<p<1$               & $4$\\
\hline
  & $\frac{1}{2}<p<0.753223$             & $3$\\
7  & $0.753223<p<0.797045$        & $0$\\
  & $0.797045<p<1$               & $4$\\
\hline
  &                     & \\
8  & $\frac{1}{2}<p<1$                    & $7$\\
  &                     & \\
\hline
  & $\frac{1}{2}<p<0.744740$             & $7$\\
9  & $0.744740<p<0.814084$        & $8$\\
  & $0.814084<p<1$               & $4$\\
\hline
 & $\frac{1}{2}<p<0.746934$             & $7$\\
10 & $0.746934<p<0.809342$        & $8$\\
 & $0.809342<p<1$               & $4$\\ \hline
\end{tabular}
\caption{Optimal transmitted moves for $4\leq k\leq 10$ in the
high-error regime.}
\label{tab:higherror}
\end{table}

These computations show that the high-error regime is considerably
more involved than the case $p\leq\frac{1}{2}$. The optimal move may change
more than once as $p$ varies, and it need not vary monotonically with
the size of the heap. Furthermore, Player I may have probability of
winning strictly smaller than $\frac{1}{2}$. Numerically,
$N_k(p)<\frac{1}{2}$ for $k=5,6,7$ approximately when
$0.682328<p<0.807286$, and for $k=9,10$ approximately when
$0.724492<p<0.826809$.

There is, however, a useful reformulation of this regime. Set
$r=(1-p)/p$, so that $0<r<1$, and let
$\overline{s}=2^n-1-s$ denote the bitwise complement of $s$ in $n$
digits. Since $d_H(\overline{s},i)=n-d_H(s,i)$, we have
\(
p^{d_H(s,i)}(1-p)^{n-d_H(s,i)}
=p^n r^{d_H(\overline{s},i)}.
\)
Consequently,
\[
N_k(p,s)=
\frac{\displaystyle\sum_{i=0}^{k-1}
(1-N_i(p))r^{d_H(\overline{s},i)}}
{\displaystyle\sum_{i=0}^{k-1}
r^{d_H(\overline{s},i)}}.
\]
Thus, although the probability of an error in each digit is larger
than $\frac{1}{2}$, the resulting distribution can be viewed as a low-error
Hamming distribution centered at the complement of the transmitted
word. This observation does not reduce the problem to the case
$p<\frac{1}{2}$, since the quantities $N_i(p)$ are still those of the
high-error game and, in general, $\overline{s}$ need not represent a
legal move. Nevertheless, it reveals an important role played by the
binary structure of the interval $\{0,\ldots,k-1\}$.

We first use this reformulation to describe completely the limiting game
as $p\to1^-$. For $0\leq s<k$, let
$d_k(s)=\min_{0\leq i<k}d_H(\overline{s},i)$ and denote by
\[
M_k(s)=\{i\in\{0,\ldots,k-1\}:
d_H(\overline{s},i)=d_k(s)\}
\]
the set of legal positions which are closest, in Hamming distance, to
the complement of the transmitted move.

\begin{prop}\label{Prop:LimitHighError}
For every $k\geq0$, the limit
\[
L_k=\lim_{p\to1^-}N_k(p)
\]
exists. Moreover, $L_0=0$ and, for every $k\geq1$,
\[
L_k=
\max_{0\leq s<k}
\frac{1}{|M_k(s)|}
\sum_{i\in M_k(s)}(1-L_i).
\]
\end{prop}

\begin{proof}
We proceed by induction on $k$. The result is immediate for $k=0$.
Suppose that the limits $L_i$ exist for every $0\leq i<k$, and fix a
transmitted move $s$. From the expression above,
\[
N_k(p,s)=
\frac{\displaystyle\sum_{i=0}^{k-1}(1-N_i(p))
r^{d_H(\overline{s},i)}}
{\displaystyle\sum_{i=0}^{k-1}r^{d_H(\overline{s},i)}}.
\]
Dividing numerator and denominator by $r^{d_k(s)}$, all terms with
$i\notin M_k(s)$ tend to zero as $p\to1^-$, while the terms
corresponding to $i\in M_k(s)$ have the same weight. Hence
\[
\lim_{p\to1^-}N_k(p,s)
=
\frac{1}{|M_k(s)|}
\sum_{i\in M_k(s)}(1-L_i).
\]
Since there are only finitely many possible transmitted moves, the
limit of their maximum is the maximum of their limits. Therefore $L_k$
exists and satisfies the stated recurrence.
\end{proof}

The sets $M_k(s)$ admit a particularly simple description.

\begin{lema}\label{Lem:NearestHighError}
Let $k=2^m+a$, where $0\leq a<2^m$, and write
$\overline{s}$ using $m+1$ binary digits. Then
$2^m-a\leq\overline{s}\leq2^{m+1}-1$. Write
$\overline{s}=x$ or $\overline{s}=2^m+x$, with $0\leq x<2^m$.
Then:
\begin{enumerate}
\item if $\overline{s}=x<2^m$, then $M_k(s)=\{x\}$;
\item if $\overline{s}=2^m+x$ and $x<a$, then
      $M_k(s)=\{2^m+x\}$;
\item if $\overline{s}=2^m+x$ and $x\geq a$, then
\[
M_k(s)=
\{x\}\cup
\{2^m+y:0\leq y<a,\ d_H(x,y)=1\}.
\]
\end{enumerate}
\end{lema}

\begin{proof}
The first two cases correspond to $\overline{s}$ itself being a legal
position, and hence it is its unique closest legal position. Suppose
then that $\overline{s}=2^m+x$ with $x\geq a$. The position $x$ is
legal and differs from $\overline{s}$ only in the leading digit, so its
Hamming distance from $\overline{s}$ is one. Any other legal position
with leading digit zero has distance $1+d_H(x,y)$ and is therefore
farther away. The remaining legal positions are precisely $2^m+y$,
with $0\leq y<a$, and their distance from $\overline{s}$ is
$d_H(x,y)$. Since $x\geq a$, none of them has distance zero. Thus
exactly those with $d_H(x,y)=1$ join $x$ in the set of closest legal
positions.
\end{proof}

We can now determine all the limiting winning probabilities.

\begin{teo}\label{Teo:LimitHighError}
The first limiting values are
$L_0=0$, $L_1=L_2=1$, $L_3=0$, and $L_k=1$ for $4\leq k\leq7$.
Moreover, for every $m\geq3$ and $2^m\leq k<2^{m+1}$,
\[
L_k=
\begin{cases}
1,
& k=2^m,\ 2^m+1
  \text{ or } 2^{m+1}-3\leq k\leq2^{m+1}-1,\\[1mm]
\dfrac12,
& 2^m+2\leq k\leq2^{m+1}-4.
\end{cases}
\]
\end{teo}

\begin{proof}
The values $L_0,\ldots,L_7$ follow directly from
Proposition~\ref{Prop:LimitHighError}. We prove the general statement
by strong induction. Write $k=2^m+a$, with $0\leq a<2^m$, and assume
that the result is known for every position smaller than $k$.

Suppose first that $a=0$. By Lemma~\ref{Lem:NearestHighError}, every
possible set $M_k(s)$ consists of a single lower $m$-digit position
$x$. Taking $x=3$, and using $L_3=0$, gives $L_k=1$. If $a=1$, we may again take $x=3$. Since $d_H(3,0)=2$, no new
position with leading digit one belongs to $M_k(s)$, and hence
$M_k(s)=\{3\}$. Thus $L_k=1$.

Let now $a=2$. Taking $x=3$ gives
$M_k(s)=\{3,2^m+1\}$. Since $L_3=0$ and, by the preceding case,
$L_{2^m+1}=1$, the corresponding winning probability tends to
$\frac{1}{2}$. On the other hand, every other possible set $M_k(s)$ consists
only of positions whose limiting value is at least $\frac{1}{2}$, with the
only possible exception of the position $3$, which has just been
considered. Indeed, in the third case of Lemma~\ref{Lem:NearestHighError}, the lower position $3$ can occur only when $x=3$; since $a=2$ the corresponding closest-position set is exactly $\{3,2^m+1\}$. No other closest-position set contains a follower with limiting value zero. Therefore no move can have limiting winning probability larger than $\frac{1}{2}$, and $L_{2^m+2}=\frac{1}{2}$.

Consider next $3\leq a\leq2^m-4$. The legal position $2^m+2$ may
occur as a singleton closest position, and since
$L_{2^m+2}=\frac{1}{2}$, this gives $L_k\geq\frac{1}{2}$. We show the opposite
inequality. If a closest legal position is obtained as a singleton,
its limiting value is at least $\frac{1}{2}$. For a set arising from the third case of Lemma~\ref{Lem:NearestHighError}, all its elements also have limiting value at least $\frac{1}{2}$, except possibly when the lower position $x=3$ occurs. Since that case requires $x\geq a$ and here $a\geq3$, necessarily $a=3$. In this exceptional case,
\(
M_k(s)=\{3,2^m+1,2^m+2\},
\)
whose limiting values are $0$, $1$ and $\frac{1}{2}$, respectively. Their
average is $\frac{1}{2}$, so the corresponding winning probability is also
$\frac{1}{2}$. Hence $L_k=\frac{1}{2}$ throughout this range.

Finally, suppose that $2^m-3\leq a<2^m$. Then the lower position $3$
belongs to the interval $[2^m-a,2^m-1]$ of possible complements.
Choosing $\overline{s}=3$ gives $M_k(s)=\{3\}$ and therefore
$L_k=1$. This completes the induction.
\end{proof}

Thus, apart from the two exceptional values $L_0=L_3=0$, the limiting
game takes only the values $\frac{1}{2}$ and $1$. In particular, for $m\geq3$
the value $L_k$ equals $1$ in a short interval around each power of
two, namely for $2^m-3\leq k\leq2^m+1$, and equals $\frac{1}{2}$ between
these intervals.

We now focus on heaps of size $2^m$, whose legal followers correspond to all $m$-bit binary words. This complete binary structure allows us to obtain stronger results on optimal moves near $p=1$ and on winning probabilities throughout the high-error regime. We first show that the move $k-1$, suggested by the computations for small powers of two, is not always optimal.

\begin{coro}\label{Coro:PowerTwoHighError}
Let $k=2^m$ with $m\geq8$. Then there exists $p_m<1$ such that, for
$p_m<p<1$, the unique optimal transmitted move is $k-4$. In
particular, transmitting $k-1$ is not always optimal for heaps whose
size is a power of two.
\end{coro}

\begin{proof}
Let $r=(1-p)/p$. When $k=2^m$, all $m$-digit words represent legal
followers. Thus a transmitted move can have limiting winning
probability one only when the complement of its $m$ lower digits is
a position $i$ with $L_i=0$. By Theorem~\ref{Teo:LimitHighError},
the only such positions are $0$ and $3$. These possibilities
correspond respectively to transmitting $k-1$ and $k-4$. Every other
move has limiting winning probability at most $\frac{1}{2}$, so for $p$
sufficiently close to one it is enough to compare these two moves.

Let $w(i)$ denote the Hamming weight of the $m$-digit representation
of $i$. For the move $k-1$, after cancelling the common factor coming
from the leading digit, we have
\[
N_k(p,k-1)=
\frac{\displaystyle\sum_{i=0}^{k-1}
(1-N_i(p))r^{w(i)}}{(1+r)^m}.
\]
The term $i=0$ contributes $1$. The positions of Hamming weight one
are the powers of two, and Theorem~\ref{Teo:LimitHighError} gives
$L_i=1$ for all of them. All remaining terms contain at least $r^2$.
Hence the numerator is $1+o(r)$, while
$(1+r)^m=1+mr+O(r^2)$. Consequently,
\[
N_k(p,k-1)=1-mr+o(r).
\]

For the move $k-4$, the distribution is centered at the lower
$m$-digit word representing $3$. Equation~\eqref{Eq:Nim} gives
$N_3(p)=2r+O(r^2)$ as $r\to0^+$. Among the $m$ neighbors of $3$
in the $m$-dimensional Hamming cube, the positions $1$, $2$ and $7$
have limiting value $1$, while the remaining $m-3$ positions are
$3+2^j$, $3\leq j\leq m-1$, and
Theorem~\ref{Teo:LimitHighError} gives limiting value $\frac{1}{2}$ for each
of them. It follows that the corresponding numerator is
$1+(m-7)r/2+o(r)$. Since the denominator is again $(1+r)^m$, we
obtain
\[
N_k(p,k-4)=1-\frac{m+7}{2}r+o(r).
\]
Therefore
\[
N_k(p,k-4)-N_k(p,k-1)
=\frac{m-7}{2}r+o(r).
\]
For $m\geq8$ this difference is positive for all sufficiently small
$r>0$, or equivalently for all $p$ sufficiently close to $1$.
Since there are only finitely many transmitted moves for this fixed $k$, and every remaining move has limiting winning probability at most $\frac{1}{2}$, a single sufficiently small interval $0<r<r_m$ excludes all of them. Hence the move $k-4$ is the unique optimal move there.
\end{proof}

For instance, the first value covered by
Corollary~\ref{Coro:PowerTwoHighError} is $k=2^8=256$. Numerically,
at $p=0.98$ one obtains
\[
N_{256}(0.98,252)\simeq0.865662,
\qquad
N_{256}(0.98,255)\simeq0.863949,
\]
in agreement with the asymptotic result.

Although the optimal move for powers of two may therefore be
non-trivial, their winning probabilities can be controlled throughout
the whole high-error regime. The following elementary estimate on
dyadic blocks is the key observation.

\begin{lema}\label{Lem:DyadicBlock}
Fix $\frac{1}{2}<p<1$, let $a=2^m$, and suppose that $N_a(p)\geq\frac{1}{2}$.
Then, for every $a\leq k<2a$,
\[
1-N_a(p)\leq N_k(p)\leq N_a(p).
\]
\end{lema}

\begin{proof}
We suppress $p$ from the notation and prove the assertion by induction
on $k=a+j$, $0\leq j<a$. The case $j=0$ is immediate. Suppose
$j\geq1$ and that the result holds for all the preceding positions in
the block.

Partition the legal followers of $a+j$ into
$L=\{0,\ldots,a-1\}$ and
$U=\{a,\ldots,a+j-1\}$. Consider any transmitted move from $a+j$,
and let $x$ be the integer represented by its last $m$ binary digits.
Conditioned on the received move belonging to $L$, the leading digit
contributes the same factor to every possible received move and
therefore cancels after normalization. The resulting distribution on
the last $m$ digits is exactly the distribution obtained by
transmitting $x$ from the position $a$. Hence the conditional
probability of winning through $L$ is at most $N_a$.

On the other hand, the inductive hypothesis gives
$N_u\geq1-N_a$ for every $u\in U$. Consequently,
$1-N_u\leq N_a$ for every such $u$, and therefore the conditional
probability of winning through $U$ is also at most $N_a$. The total
winning probability of any transmitted move is a convex combination
of these two conditional probabilities, and hence $N_{a+j}\leq N_a$.

For the reverse inequality, let $s^*<a$ be an optimal transmitted
move from $a$, and transmit the same move $s^*$ from $a+j$.
Conditioned on the received move belonging to $L$, the winning
probability is exactly $N_a$. Moreover, by the upper inequality
already established inductively, $N_u\leq N_a$ for every $u\in U$,
so $1-N_u\geq1-N_a$. Thus the conditional winning probability
through $U$ is at least $1-N_a$. Since $N_a\geq\frac{1}{2}$, we have
$N_a\geq1-N_a$, and the transmitted move $s^*$ guarantees a winning
probability of at least $1-N_a$. Therefore
$N_{a+j}\geq1-N_a$, completing the induction.
\end{proof}

The dyadic block estimate yields a complete answer to the question of
whether Player I has probability at least $\frac{1}{2}$ from a heap whose size
is a power of two.

\begin{teo}\label{Teo:PowerTwoHighError}
Let $\frac{1}{2}<p<1$. For every $m\geq0$,
\[
N_{2^m}(p)
\geq
\frac12+\frac12(2p-1)^m
>\frac12.
\]
Moreover,
\[
(2p-1)\left(N_{2^m}(p)-\frac12\right)
\leq
N_{2^{m+1}}(p)-\frac12
\leq
N_{2^m}(p)-\frac12.
\]
In particular, the sequence $\bigl(N_{2^m}(p)\bigr)_{m\geq0}$ is
non-increasing and convergent.
\end{teo}

\begin{proof}
We first prove the one-step estimates. Let $a=2^m$ and assume
$N_a\geq\frac{1}{2}$. Let $s^*<a$ be an optimal move from $a$. Consider the position $2a$. All its legal followers have zero as their
most significant binary digit. Thus, after conditioning on receiving
a valid move, this digit cancels from the distribution and we may
consider the remaining $m+1$ digits. Transmit the move $a+s^*$. In these $m+1$ digits this word consists
of a leading $1$ followed by the $m$ digits of $s^*$. With probability
$p$ the leading $1$ is changed into $0$, in which case the received
position belongs to $\{0,\ldots,a-1\}$. Conditioned on this event,
the remaining $m$ digits give exactly the distribution corresponding
to the optimal move $s^*$ from $a$, so the conditional probability of
winning is $N_a$. With probability $1-p$ this digit remains $1$, in which case the
received position is $a+j$ for some $0\leq j<a$. By
Lemma~\ref{Lem:DyadicBlock}, $N_{a+j}\leq N_a$, and therefore the
probability of winning after reaching such a position is at least
$1-N_a$. Hence
\[
N_{2a}\geq pN_a+(1-p)(1-N_a)
=\frac12+(2p-1)\left(N_a-\frac12\right).
\]
This proves the first one-step inequality.

For the upper bound, consider an arbitrary transmitted move from
$2a$. Conditioned on the received move belonging to the lower block
$\{0,\ldots,a-1\}$, its winning probability is the value of some
transmitted move from $a$ and is therefore at most $N_a$. Conditioned
on the received move belonging to the upper block
$\{a,\ldots,2a-1\}$, Lemma~\ref{Lem:DyadicBlock} gives
$N_{a+j}\geq1-N_a$, and hence $1-N_{a+j}\leq N_a$ for every possible
received position. Thus the conditional winning probability through
the upper block is also at most $N_a$. It follows that
$N_{2a}\leq N_a$. Starting with $N_1=1$ and applying the lower estimate recursively gives
\[
N_{2^m}(p)-\frac12
\geq
(2p-1)^m\left(N_1-\frac12\right)
=
\frac12(2p-1)^m,
\]
which proves the stated lower bound. The upper one-step estimate shows
that $\bigl(N_{2^m}(p)\bigr)_{m\geq0}$ is non-increasing. Since it is
bounded below by $\frac{1}{2}$, it is convergent.
\end{proof}

Combining Theorem~\ref{Teo:PowerTwoHighError} with the result proved
above for $0\leq p\leq\frac{1}{2}$ gives the following complete statement.

\begin{coro}\label{Coro:PowerTwoAllP}
For every $m\geq0$ and every $0\leq p<1$,
\[
N_{2^m}(p)\geq\frac12.
\]
For $\frac{1}{2}<p<1$ the inequality is strict.
\end{coro}

Thus the winning-probability assertion for powers of two holds
throughout the whole range of parameters for which the game is
defined. On the other hand, Corollary~\ref{Coro:PowerTwoHighError}
shows that the corresponding optimal move does not admit the simple
description suggested by the first numerical examples.

There is one further structural property of the functions $N_k(p)$
which is useful when considering the high-error regime.

\begin{prop}\label{Prop:PiecewiseRational}
For every fixed $k$, the functions $N_k(p,s)$, $0\leq s<k$, and
$N_k(p)$ are continuous and piecewise rational functions of $p$ on
$(0,1)$, with a finite number of pieces. Consequently, the set of
optimal transmitted moves can change only at finitely many values of
$p$.
\end{prop}

\begin{proof}
We argue by induction on $k$. The assertion is immediate for $k=0,1$.
Assume that it holds for every $i<k$. On a common finite partition of
$(0,1)$, all the functions $N_i(p)$ are rational. Equation
\eqref{Eq:formularNim} then shows that, for each fixed $s$, the function
$N_k(p,s)$ is rational on every interval of this partition, since its
denominator is strictly positive for $0<p<1$.

On each such interval, the difference between the values corresponding
to two transmitted moves is a rational function. Unless the two
functions coincide identically, they can be equal only at finitely
many points. Refining the partition by all these points gives a finite
partition on which the maximum is represented by one or more fixed
rational functions. Hence $N_k(p)$ is piecewise rational. Continuity
follows from Equation~\eqref{Eq:formularNim} and from the fact that
the maximum of finitely many continuous functions is continuous.
\end{proof}

The preceding results settle the question of whether the first player has winning probability at least $\frac{1}{2}$ for heaps of size a power of two, throughout $0\leq p<1$, and give a complete description of the limiting game as $p\to1^-$. Nevertheless, determining the optimal strategy throughout $\frac{1}{2}<p<1$ remains considerably more involved, as already indicated by Table~\ref{tab:higherror}.

\begin{preg}\label{Question:HighError}
Describe the intervals of optimality of the transmitted moves for
$\frac{1}{2}<p<1$ as a function of $k$. In particular, characterize the
transition values of $p$ at which the optimal move changes and determine
for which pairs $(k,p)$ one has $N_k(p)\geq\frac{1}{2}$.
\end{preg}

Beyond the specific problem above, the general framework introduced in
Section~\ref{Sec:Framework} suggests several other directions. It would
be natural to study noisy versions of multi-pile \emph{Nim} and other
subtraction games, as well as more general error mechanisms for
\emph{Chomp!}. Another question is to characterize those combinatorial
games that admit a move error function turning them into fair or unfair
chance games. We leave these problems for future work.

\subsection*{Acknowledgements.} The authors used ChatGPT (OpenAI) to assist with
English-language editing, improvements in mathematical
exposition, notation consistency, and the discussion
and verification of selected mathematical arguments.
All mathematical results and proofs have been
independently checked by the authors, who take
full responsibility for the content of the manuscript.


\begin{thebibliography}{99}

\bibitem{ArchibaldShoham}
C. Archibald and Y. Shoham.
Modeling billiards games.
In \textit{Proceedings of the 8th International Joint Conference on
Autonomous Agents and Multiagent Systems (AAMAS 2009)}, Vol. 1,
pp. 193--199. IFAAMAS (2009).

\bibitem{BeckEjovFilar}
J. D. Beck, V. Ejov and J. A. Filar.
Incompetence and impact of training in bimatrix games.
\textit{Automatica} 48 (2012), 2400--2408.

\bibitem{BeckFilar}
J. D. Beck and J. A. Filar.
Games, incompetence, and training.
In \textit{Advances in Dynamic Game Theory}, pp. 93--110.
Birkh\"auser Boston (2007).

\bibitem{Conway} E. R. Berlekamp, J. H. Conway and R. K. Guy. \textit{Winning Ways for Your Mathematical Plays}. AK Peters (2001).

\bibitem{FarinaGatti}
G. Farina and N. Gatti.
Extensive-form perfect equilibrium computation in two-player games.
In \textit{Proceedings of the Thirty-First AAAI Conference on
Artificial Intelligence (AAAI 2017)}, pp. 502--508 (2017).

\bibitem{Ferguson} T. S. Ferguson. \textit{A Course in Game Theory}. World Scientific (2020).

\bibitem{Krawec}
W. O. Krawec.
$n$-Player impartial combinatorial games with random players.
\textit{Theoret. Comput. Sci.} 569 (2015), 1--12.

\bibitem{GrahamKleshninaFilar}
T. Graham, M. Kleshnina and J. A. Filar.
Where do mistakes lead? A survey of games with incompetent players.
\textit{Dynamic Games Appl.} 13 (2023), 231--264.


\bibitem{Maby2012}
E. Maby, M. Perrin, O. Bertrand, G. Sanchez and J. Mattout.
BCI could make old two-player games even more fun:
a proof of concept with ``Connect Four''.
\textit{Advances in Human-Computer Interaction}
2012 (2012), Article ID 124728, 8 pp.


\bibitem{NimVariants}
M. van den Bergh, W. A. Kosters and F. M. Spieksma.
Nim variants.
\textit{ICGA J.} 44 (2022), 2--17.

\bibitem{Shapley}
L. S. Shapley.
Stochastic games.
\textit{Proc. Natl. Acad. Sci. USA} 39 (1953), 1095--1100.

\bibitem{West} D. B. West. \textit{Introduction to Graph Theory},
2nd ed. Prentice Hall, Upper Saddle River, NJ (2001).

\end{thebibliography}
\end{document}